\documentclass[a4paper]{amsart}
\usepackage[latin1]{inputenc}
\usepackage{amssymb}
\usepackage{amsmath}
\usepackage{mathrsfs}
\usepackage{eufrak}
\usepackage{amsthm}
\usepackage{amsfonts}
\usepackage{textcomp}
\usepackage{graphicx}
\usepackage[pdftex]{color}
\usepackage{paralist}
\usepackage[shortlabels]{enumitem}
\usepackage{hyperref}
\usepackage{comment}
\usepackage[arrow, matrix, curve]{xy}
\usepackage{tikz}

\newtheorem*{corollary*}{Corollary}
\newtheorem*{conjecture*}{Conjecture}
\newtheorem*{example*}{Example}
\newtheorem*{theorem*}{Theorem}
\newtheorem*{proposition*}{Proposition}

\newtheorem{theorem}{Theorem}[section]

\newtheorem{corollary}[theorem]{Corollary}
\newtheorem{lemma}[theorem]{Lemma}

\newtheorem{problem}[theorem]{Problem}

\newtheorem*{claim*}{Claim}

\theoremstyle{definition}

\newtheorem{example}[theorem]{Example}

\theoremstyle{remark}

\numberwithin{equation}{section}

\makeatletter
\renewcommand*\env@matrix[1][\
arraystretch]{%
  \edef\arraystretch{#1}%
  \hskip -\arraycolsep
  \let\@ifnextchar\new@ifnextchar
  \array{*\c@MaxMatrixCols c}}
\makeatother

\renewcommand{\mod}{\operatorname{mod}}

\newcommand{\Ext}{\operatorname{Ext}}

\newcommand{\CM}{\operatorname{CM}}

\newcommand{\Hom}{\operatorname{Hom}}

\begin{document}

\title{On self-orthogonal modules in Iwanaga-Gorenstein rings}
\date{\today}



\author{Ren\'{e} Marczinzik}
\address[Marczinzik]{Mathematical Institute of the University of Bonn, Endenicher Allee 60, 53115 Bonn, Germany}
\email{marczire@math.uni-bonn.de}
\begin{abstract}
Let $A$ be an Iwanaga-Gorenstein ring. Enomoto conjectured that a self-orthogonal $A$-module has finite projective dimension. We prove this conjecture for $A$ having the property that every indecomposable non-projective maximal Cohen-Macaulay module is periodic. This answers a question of Enomoto and shows the conjecture for monomial quiver algebras and hypersurface rings.
\end{abstract}

\maketitle

\section*{Introduction}
We assume always that $A$ is a two-sided noetherian semiperfect ring and all modules are finitely generated right modules unless otherwise stated.
Recall that $A$ is called \emph{$n$-Iwanaga-Gorenstein} if the injective dimensions of $A$ as a left and right module are equal to $n$. If the $n$ does not matter we will often just say Iwanaga-Gorenstein ring instead of $n$-Iwanaga-Gorenstein ring. The category of \emph{maximal Cohen-Macaulay modules} $\CM A$ of a $n$-Iwanaga-Gorenstein ring is defined as the category of n-th syzygy modules $\Omega^n(\mod A)$ consisting of modules $X$ that are projective or direct summands of a module of the form $\Omega^n(M)$ for some $M \in \mod A$.
A module $M$ is called \emph{self-orthogonal} if $\Ext_A^i(M,M)=0$ for all $i \geq 1$.
The definition of Iwanaga-Gorenstein rings includes the classical cases of Iwanaga-Gorenstein rings, namely the commutative local Gorenstein rings and the Iwanaga-Gorenstein Artin algebras. 

We are interested in the following problem that was stated in \cite{E} as conjecture 4.8 for Artin algebras. 
\begin{problem}
Let $A$ be Iwanaga-Gorenstein and let $M$ be self-orthogonal. Then $M$ has finite projective dimension.

\end{problem}

A positive solution of this problem would have important consequences for the theory of tilting modules for Iwanaga-Gorenstein Artin algebras, see section 3 and 4 in \cite{E}.
For Artin algebras the conjecture of Enomoto is a generalisation of the classical Tachikawa conjecture that states that a self-orthogonal module over a selfinjective algebra is projective. The Tachikawa conjecture can be seen as one of the most important homological conjectures for Artin algebras since a counterexample to the Tachikawa conjecture would give counterexamples to other homological conjectures such as the Nakayama conjecture, the Auslander-Reiten conjecture and the finitistic dimension conjecture, see for example \cite{Y} for a survey on those conjectures.
Our main result gives a positive answer to the above problem for an important class of Iwanaga-Gorenstein algebras:
\begin{theorem} 
Let $A$ be an $n$-Iwanaga-Gorenstein ring such that every indecomposable non-projective module $X \in \CM A$ is periodic. Assume $M$ has the property that $\Ext_A^u(M,M)=0$ for all $u>n$. Then $M$ has finite projective dimension. In particular, all self-orthogonal modules have finite projective dimension. 

\end{theorem}

In \cite{E} a positive solution to the above problem was proven for representation-finite Iwanaga-Gorenstein Artin algebras using the theory of generalised tilting modules. In question 4.7 of \cite{E} a more direct proof for this case is asked for and our main result gives such a direct proof for a much larger class of Iwanaga-Gorenstein Artin algebras, which contain all CM-finite Iwanaga-Gorenstein rings and in particular the subclass of all such representation-finite algebras.

\section{Proof of the main results}
In this section $\underline{\CM} A$ will denote the category of maximal Cohen-Macaulay modules modulo projective modules.

\begin{lemma} \label{mainlemma}
Assume $A$ is an $n$-Iwanaga-Gorenstein ring and $M \in \CM A$ and $N \in \mod A$.
\begin{enumerate}
\item $\underline{\Hom}_A(M,N) \cong \underline{\Hom}_A(\Omega^s(M),\Omega^s(N))$ for all $s \geq 0$.
\item $\Ext_A^p(M,N) \cong \underline{\Hom}_A(\Omega^p(M),N)$ for all $p \geq 1$.
\item If $M$ is indecomposable, then also $\Omega^i(M) \in \CM A$ is indecomposable for all $i \geq 1$.
\item If $M, N \in \CM A$ are indecomposable and satisfy $\Omega^1(M) \cong \Omega^1(N)$ then $M \cong N$.
\end{enumerate}

\end{lemma}
\begin{proof}
Note that every module $X \in \CM A$ satisfies $\Ext_A^i(X,A)=0$ for all $i>0$, since $X$ is a direct summand of a module of the form $\Omega^n(Y)$ and $\Ext_A^i(\Omega^n(Y),A)=\Ext_A^{n+i}(Y,A)=0$ for all $i>0$ since $A$ has injective dimension $n$.
Then (1) and (2) are a special case of \cite[2.1]{I}. (3) follows from \cite[Corollary 3.3]{RZ} and (4) is a consequence \cite[Theorem 5.5]{KKN}.

\end{proof}

Recall that a module $X \in \mod A$ is called \emph{periodic} if $\Omega^l(X) \cong X$ for some $l \geq 1$.

\begin{theorem} \label{maintheorem}
Let $A$ be an $n$-Iwanaga-Gorenstein ring such that every indecomposable non-projective module $X \in \CM A$ is periodic. Assume $M$ has the property that $\Ext_A^u(M,M)=0$ for all $u>n$. Then $M$ has finite projective dimension. In particular, all self-orthogonal modules have finite projective dimension. 

\end{theorem}

\begin{proof}
Assume $M$ has the property that $\Ext_A^{n+l}(M,M)=0$ for all $l \geq 1$.
Then 
$$\Ext_A^{n+l}(M,M) \cong \Ext_A^l(\Omega^n(M),M) \cong \underline{\Hom}_A(\Omega^{l+n}(M),M)$$ by dimension shifting and \ref{mainlemma} (2), which we are allowed to use since $\Omega^n(M) \in \CM A$.
Using \ref{mainlemma} (1) with $s=n$ and then (2) again we obtain:
$$\underline{\Hom}_A(\Omega^{l+n}(M),M)\cong \underline{\Hom}_A(\Omega^{l+n+n}(M),\Omega^n(M))\cong\Ext_A^{n+l}(\Omega^n(M), \Omega^n(M)).$$
Thus $\Ext_A^{n+l}(\Omega^n(M), \Omega^n(M))=0$ for all $l \geq 1$, since we assume that $\Ext_A^{n+l}(M,M)=0$ for all $l \geq 1$. 
Let $X$ be an indecomposable direct summand of $\Omega^n(M)$. Then $X \in \CM A$ with 
$\Ext_A^{n+l}(X, X)=0$ for all $l \geq 1$. Assume that $X$ is non-zero.
By assumption $X$ is periodic. So assume that $X \cong \Omega^q(X)$ for some $q \geq 1$. Note that this also implies that $X \cong \Omega^{qm}(X)$ for all $m \geq 1$.
Then for all $p \geq 1$ and $m \geq 1$ we obtain:
$$\Ext_A^p(X,X) \cong \Ext_A^p(\Omega^{qm}(X),X) \cong \Ext_A^{p+qm}(X,X).$$
Now choose $m$ big enough so that $p+qm >n$.
Then 
$$\Ext_A^p(X,X)\cong\Ext_A^{p+qm}(X,X)=0.$$
Thus $X$ is self-orthogonal.
But we also have by \ref{mainlemma}(2)
$$\Ext_A^q(X,X) \cong \underline{\Hom}_A(\Omega^q(X),X) \cong \underline{\Hom}_A(X,X) \neq 0,$$
since the identity map in $\underline{\Hom}_A(X,X)$ is certainly non-zero.
This is a contradiction and thus $X$ must be zero.
Thus every indecomposable direct summand of $\Omega^n(M)$ is the zero module and thus $\Omega^n(M)$ itself must be the zero module, which implies that $\Omega^n(M)=0$ and $M$ has finite projective dimension.
\end{proof}

Recall that an Iwanaga-Gorenstein ring is called \emph{CM-finite} if there are only finitely many indecomposable modules in $\CM A$.
\begin{corollary}
Let $A$ be a CM-finite Iwanaga-Gorenstein ring.
Then every self-orthogonal module has finite projective dimension.

\end{corollary}

\begin{proof}
We show that every indecomposable module $X \in \CM A$ is periodic. Then the result follows from \ref{maintheorem}.
Let $X$ be indecomposable. Since with $X$ also $\Omega^i(X) \in \CM A$ is indecomposable for all $i \geq 0$ by \ref{mainlemma} (3) and since there are only finitely many indecomposable modules in $\CM A$, we have that
$\Omega^i(X) \cong \Omega^{i+l}(X)$ for some $i \geq 0$ and $l \geq 1$.
This implies that $X \cong \Omega^l(X)$ by \ref{mainlemma} (4) and thus $X$ is periodic.
\end{proof}
We give two important examples. The first is for finite dimensional algebras and the second for commutative local rings.
\begin{example}
Let $A$ be a finite dimensional quiver algebra $KQ/I$ with admissible monomial ideal $I$.
Then $A$ has the property that $\Omega^2(\mod-A)$ is representation-finite, see \cite{Z}, and thus $A$ is CM-finite if $A$ is Iwanaga-Gorenstein. In particular, all gentle algebras are CM-finite Iwanaga-Gorenstein algebras as gentle quiver algebras are always Iwanaga-Gorenstein by \cite{GR}.
Thus for the class of monomial Iwanaga-Gorenstein algebras, every self-orthogonal module has finite projective dimension by our main result.

\end{example}

\begin{example}
Let $R$ be a regular commutative local ring and $f \neq 0$.
Then the hypersurface ring $A=R/(f)$ is Iwanaga-Gorenstein and every module $X \in \CM A$ is periodic of period at most 2 by the classical result about matrix factorisations by Eisenbud, see \cite{Ei}. By our main result, every self-orthogonal $A$-module has finite projective dimension.
\end{example}

\end{document}